\documentclass[11pt, a4paper]{article}
\usepackage{}

\usepackage{indentfirst}
\usepackage{lineno}
\usepackage{graphicx}
\usepackage{ae}
\usepackage{amsmath}
\usepackage{amssymb}
\usepackage{latexsym}
\usepackage{url}
\usepackage{epsfig}
\usepackage{cite}
\usepackage{mathrsfs}
\usepackage{amsfonts}
\usepackage{amsthm}

\def\qed{\hfill$\Box$\vspace{12pt}}

\usepackage{color}

\long\def\delete#1{}

\newcommand{\bmat}[1]{\begin{bmatrix}#1\end{bmatrix}}
\newcommand{\pmat}[1]{\begin{pmatrix}#1\end{pmatrix}}

\newcommand{\be}{\begin{equation}}
\newcommand{\ee}{\end{equation}}
\newcommand{\ben}{\begin{equation*}}
\newcommand{\een}{\end{equation*}}
\newcommand{\bea}{\begin{eqnarray}}
\newcommand{\eea}{\end{eqnarray}}
\newcommand{\bean}{\begin{eqnarray*}}
\newcommand{\eean}{\end{eqnarray*}}

\def\det{{\rm det}}

\newtheorem{thm}{Theorem}[section]
\newtheorem{cor}[thm]{Corollary}

\newtheorem{defn}[thm]{Definition}

\numberwithin{equation}{section}

\title{Spectra of the subdivision-vertex and subdivision-edge coronae}

\author{Pengli Lu\; and\; Yufang Miao
\\
{\small School of Computer and Communication}\\[-0.8ex]
{\small Lanzhou University of Technology}\\[-0.8ex]
{\small Lanzhou, 730050, Gansu, P.R. China}\\
\emph{{\small \tt lupengli88@163.com, miaoyufanghappy@163.com }} }

\date{}

\begin{document}

\openup 0.5\jot
\maketitle
%\linenumbers

\begin{abstract}
The subdivision graph $\mathcal{S}(G)$ of a graph $G$ is the graph obtained by inserting a new vertex into every edge of $G$. Let $G_1$ and $G_2$ be two vertex disjoint graphs. The \emph{subdivision-vertex corona} of $G_1$ and $G_2$, denoted by $G_1\odot G_2$, is the graph obtained from $\mathcal{S}(G_1)$ and $|V(G_1)|$ copies of $G_2$, all vertex-disjoint, by joining the $i$th vertex of $V(G_1)$ to every vertex in the $i$th copy of $G_2$. The \emph{subdivision-edge corona} of $G_1$ and $G_2$, denoted by $G_1\circleddash G_2$, is the graph obtained from $\mathcal{S}(G_1)$ and $|I(G_1)|$ copies of $G_2$, all vertex-disjoint, by joining the $i$th vertex of $I(G_1)$ to every vertex in the $i$th copy of $G_2$, where $I(G_1)$ is the set of inserted vertices of $\mathcal{S}(G_1)$. In this paper we determine the adjacency spectra, the Laplacian spectra and the signless Laplacian spectra of $G_1\odot G_2$ (respectively, $G_1\circleddash G_2$) in terms of the corresponding spectra of $G_1$ and $G_2$. As applications, the results on the spectra of $G_1\odot G_2$ (respectively, $G_1\circleddash G_2$) enable us to construct infinitely many pairs of cospectral graphs. The adjacency spectra of $G_1\odot G_2$ (respectively, $G_1\circleddash G_2$) help us to construct many infinite families of integral graphs. By using the Laplacian spectra, we also obtain the number of spanning trees and Kirchhoff index of $G_1\odot G_2$ and $G_1\circleddash G_2$, respectively.

\bigskip

\noindent\textbf{Keywords:}  Spectrum, Cospectral graphs, Integral graphs, Spanning trees, Kirchhoff index, Subdivision-vertex
corona, Subdivision-edge corona

\bigskip

\noindent{{\bf AMS Subject Classification (2010):} 05C50}
\end{abstract}
%\linenumbers

\section{Introduction}

We only consider undirected and simple graphs throughout this
paper. Let $G=(V(G),E(G))$ be a graph with vertex set
$V(G)=\{v_1,v_2,\ldots,v_n\}$ and edge set $E(G)$. The
\emph{adjacency matrix} of $G$, denoted by $A(G)$, is the $n \times
n$ matrix whose $(i,j)$-entry is $1$ if $v_i$ and $v_j$ are adjacent in $G$ and $0$ otherwise. Let $d_G(v_i)$ be the degree of vertex $v_i$ in $G$. Denote $D(G)$ to be the diagonal matrix with diagonal entries $d_G(v_1),\ldots,d_G(v_n)$. The \emph{Laplacian matrix} of $G$ and the \emph{signless Laplacian matrix} of $G$ are defined as $L(G)=D(G)-A(G)$ and $Q(G)=D(G)+A(G)$, respectively. Let $\phi(A(G);x)=\det(xI_n-A(G))$, or simply $\phi(A(G))$, be the \emph{adjacency characteristic polynomial} of $G$. Similarly,  denote $\phi(L(G))$ (respectively, $\phi(Q(G))$) as the \emph{Laplacian} (respectively, \emph{signless Laplacian}) \emph{characteristic polynomial}. Denote the eigenvalues of $A(G), L(G)$ and $Q(G)$ by $\lambda_1(G)\geq\lambda_2(G)\geq\cdots\geq\lambda_n(G)$, $0=\mu_1(G)\leq\mu_2(G)\leq\cdots\leq\mu_n(G)$, $\nu_1(G)\leq\nu_2(G)\leq\cdots\leq\nu_n(G)$, respectively. The eigenvalues (together with the multiplicities) of $A(G), L(G)$ and $Q(G)$  are called the \emph{$A$-spectrum},  \emph{$L$-spectrum} and  \emph{$Q$-spectrum} of $G$, respectively. Graphs with the same $A$-spectra (respectively, $L$-spectra, $Q$-spectra) are called \emph{$A$-cospectral} (respectively, \emph{$L$-cospectral}, \emph{$Q$-cospectral}) graphs. For more terminologies not defined here, the readers can refer to  \cite{kn:Cvetkovic95,kn:Cvetkovic10,kn:Brouwer12}.

The \emph{corona} of two graphs was first introduced by \textsc{R. Frucht} and \textsc{F. Harary} in \cite{kn:Frucht70} with the goal of constructing a graph whose automorphism group is the wreath product of the two component automorphism groups. It is known that the $A$-spectra (respectively, $L$-spectra, $Q$-spectra) of the corona of any two graphs can be expressed by that of the two factor graphs  \cite{kn:Barik07,kn:Cui12,kn:McLeman11,kn:Wang12}. Similarly, the $A$-spectra (respectively, $L$-spectra, $Q$-spectra) of the \emph{edge corona} \cite{kn:Hou10} of two graphs, which is a variant of the corona operation, were completely computed in \cite{kn:Cui12,kn:Hou10,kn:Wang12}. Another variant of the corona operation, the \emph{neighbourhood corona}, was introduced in \cite{kn:Gopalapillai11} recently. In \cite{kn:Gopalapillai11}, the $A$-spectrum of the neighbourhood corona of an arbitrary graph and a regular graph was given in terms of that of the two factor graphs. The author also gave the $L$-spectrum of the neighbourhood corona of a regular graph and an arbitrary graph.

The \emph{subdivision graph} $\mathcal{S}(G)$ of a graph $G$ is the graph obtained by inserting a new vertex into every edge of $G$ \cite{kn:Cvetkovic10}. We denote the set of such new vertices by $I(G)$. In \cite{kn:Indulal12}, two new graph operations based on subdivision graphs: \emph{subdivision-vertex join} and \emph{subdivision-edge join} were introduced, and the $A$-spectra of subdivision-vertex join (respectively, subdivision-edge join) of two regular graphs were computed in terms of that of the two graphs. More work on their $L$-spectra and $Q$-spectra were presented in \cite{kn:Liu12}. Meanwhile, \textsc{X. Liu} and \textsc{P. Lu} defined two new graph operations based on subdivision graphs: \emph{subdivision-vertex neighbourhood corona} and \emph{subdivision-edge neighbourhood corona} \cite{kn:LiuLu12}. The authors determined the $A$-spectrum, the $L$-spectrum and the $Q$-spectrum of the subdivision-vertex neighbourhood corona (respectively, subdivision-edge neighbourhood corona) of a regular graph and an arbitrary graph.

Motivated by the work above, we define two new graph operations based on subdivision graphs as follows.

\begin{defn}
\label{defn1}
{\em The \emph{subdivision-vertex corona} of two vertex-disjoint graphs $G_1$ and $G_2$, denoted by
$G_1\odot G_2$, is the graph obtained from $\mathcal{S}(G_1)$ and
$|V(G_1)|$ copies of $G_2$, all vertex-disjoint, by joining the
$i$th vertex of $V(G_1)$ to every vertex in the $i$th copy of $G_2$.}
\end{defn}

\begin{defn}
\label{defn2}
{\em The \emph{subdivision-edge corona} of two vertex-disjoint graphs $G_1$ and $G_2$, denoted by
$G_1\circleddash G_2$, is the graph obtained from $\mathcal{S}(G_1)$
and $|I(G_1)|$ copies of $G_2$, all vertex-disjoint, by joining the $i$th vertex of $I(G_1)$ to every vertex in the $i$th copy of $G_2$.}
\end{defn}

Let $P_n$ denote a path on $n$ vertices. Figure \ref{gg2} depicts the subdivision-vertex corona $P_4\odot P_2$ and subdivision-edge corona $P_4\circleddash P_2$, respectively.
\begin{figure}
\centering
\vspace{-0.9cm}
\includegraphics*[height=7.8cm]{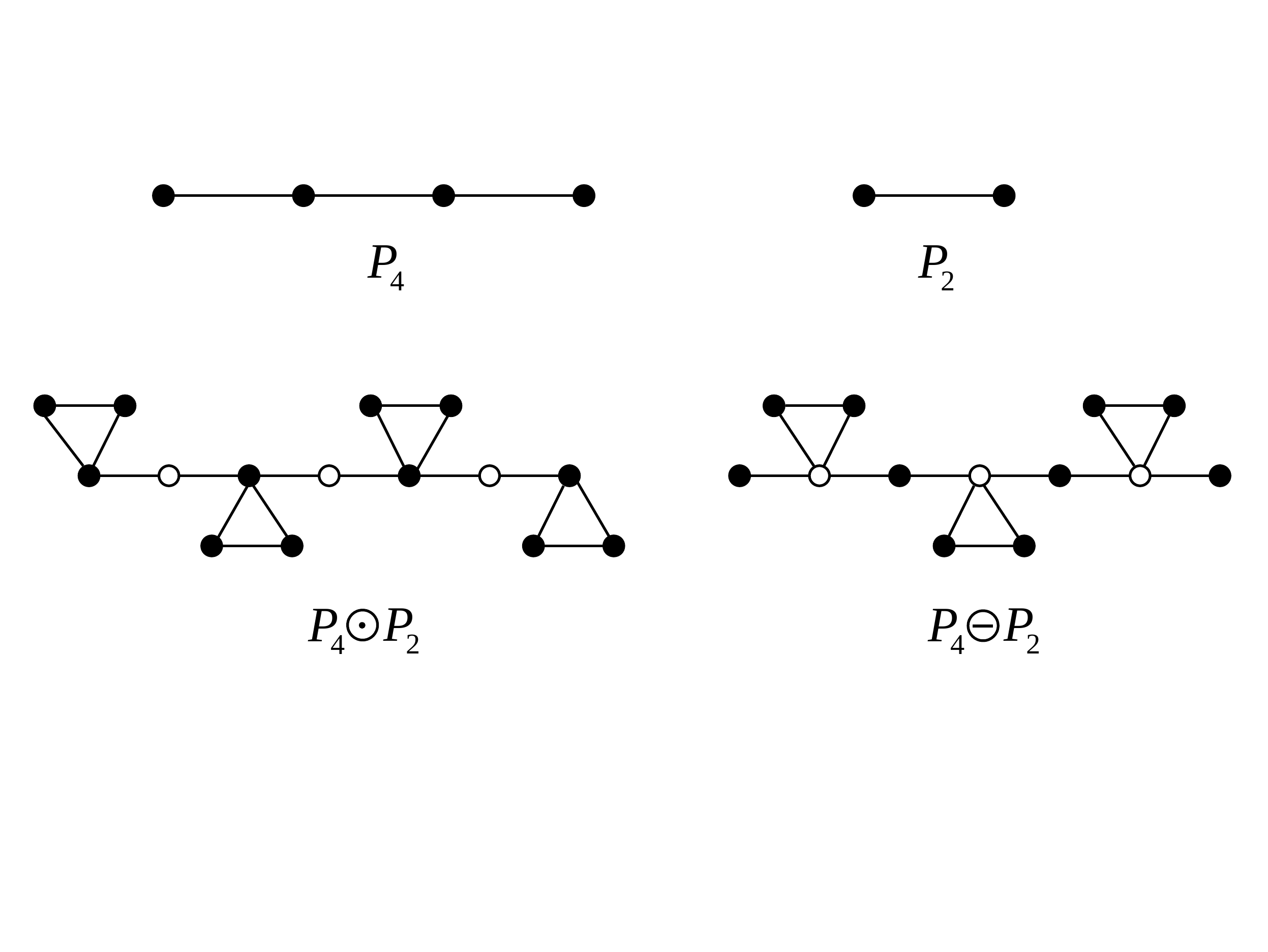}
\vspace{-2.6cm}
\caption{\small An example of subdivision-vertex and subdivision-edge coronae.}
\label{gg2}
\end{figure}
Note that if $G_1$ is a graph on $n_1$ vertices and $m_1$ edges and $G_2$ is a graph on $n_2$ vertices and $m_2$ edges, then the subdivision-vertex corona $G_1\odot G_2$ has $n_1(1+n_2)+m_1$ vertices and $2m_1+n_1(n_2+m_2)$ edges, and the subdivision-edge corona $G_1\circleddash G_2$
has $m_1(1+n_2)+n_1$ vertices and $m_1(2+n_2+m_2)$ edges.

In this paper, we will determine the $A$-spectra, the $L$-spectra and the $Q$-spectra of $G_1\odot
G_2$ (respectively, $G_1\circleddash G_2$) with the help of the
\emph{coronal} of a matrix and the \emph{Kronecker product}. The
\emph{$M$-coronal} $\Gamma_M(x)$ of an $n\times n$ matrix $M$ is defined
\cite{kn:McLeman11,kn:Cui12} to be the sum of the entries of the
 matrix $(x I_n-M)^{-1}$, that is,
$$\Gamma_M(x)=\mathbf{1}_n^T(x I_n-M)^{-1}\mathbf{1}_n. $$
where $\mathbf{1}_n$ denotes the column vector of size $n$ with all
the entries equal one. It is well known \cite[Proposition 2]{kn:Cui12} that, if $M$ is an
$n \times  n$ matrix with each row sum equal to a constant $t$, then
\be \label{eq:GammaT} \Gamma_{M}(x) = \frac{n}{x-t}. \ee In
particular, since for any graph $G_2$ with $n_2$ vertices, each row
sum of $L(G_2)$ is equal to $0$, we have \be \label{eq:GammaTL}
\Gamma_{L(G_2)}(x) = n_2/x. \ee

The \emph{Kronecker product} $A\otimes B$ of two matrices
$A=(a_{ij})_{m \times n}$ and $B=(b_{ij})_{p \times q}$ is the $mp
\times nq$ matrix obtained from $A$ by replacing each element
$a_{ij}$ by $a_{ij}B$. This is an associative operation with the
property that $(A\otimes B)^T=A^T\otimes B^T$ and $(A\otimes B)(C\otimes D)=AC\otimes BD$ whenever the products $AC$ and $BD$ exist. The latter implies $(A\otimes B)^{-1}=A^{-1}\otimes B^{-1}$ for nonsingular matrices $A$ and $B$. Moreover, if $A$ and $B$ are $n \times n$ and $p \times p$ matrices, then $\det(A\otimes B)=(\det A)^p (\det B)^n$. The reader is referred to \cite{kn:Kronecker} for other properties of the Kronecker product not mentioned here.

The paper is organized as follows. In Section \ref{SVC:spec}, we compute the $A$-spectra, the $L$-spectra and the $Q$-spectra of the subdivision-vertex corona $G_1\odot G_2$ for a regular graph $G_1$ and an arbitrary graph $G_2$ (see Theorems
\ref{CoronaTh1}, \ref{CoronaTh2}, \ref{CoronaTh3}). Section \ref{SEC:spec} mainly investigates the $A$-spectra, the $L$-spectra and the $Q$-spectra of the subdivision-edge corona $G_1\circleddash G_2$ for a
regular graph $G_1$ and an arbitrary graph $G_2$ (see Theorems \ref{CoronaTh4},
\ref{CoronaTh5}, \ref{CoronaTh6}). As we will see in Corollaries \ref{ACorRegular3}, \ref{ACorRegular6}, \ref{ACorRegular8}, \ref{ACorRegular11}, \ref{ACorRegular14}, \ref{ACorRegular16}, our results on the spectra of $G_1\odot G_2$ and $G_1\circleddash G_2$ enable us to construct infinitely many pairs of cospectral graphs. Our constructions of infinite families of integral graphs are stated in Corollaries \ref{ACorRegular17}, \ref{ACorRegular18d}, \ref{ACorRegular18fd}, \ref{ACorRegular91fd}. In Corollaries \ref{ACorRegular5}, \ref{ACorRegular51}, \ref{ACorRegular13}, \ref{ACorRegular133}, we compute the number of spanning trees and the Kirchhoff index of $G_1\odot G_2$ (respectively, $G_1\circleddash G_2$) for a regular graph $G_1$ and an arbitrary graph $G_2$.

\section{Spectra of subdivision-vertex coronae}\label{SVC:spec}

Let $G_1$ be an arbitrary graphs on $n_1$ vertices and $m_1$ edges,
and $G_2$ an arbitrary graphs on $n_2$ vertices, respectively. We
first label the vertices of $G_1\odot G_2$ as follows. Let
$V(G_1)=\left\{v_1,v_2,\ldots,v_{n_1}\right\}$,
$I(G_1)=\left\{e_1,e_2,\ldots,e_{m_1}\right\}$ and
$V(G_2)=\left\{u_1,u_2,\ldots,u_{n_2}\right\}$. For $i = 1, 2,
\ldots, n_1$, let $u_1^i,u_2^i,\ldots,u_{n_2}^i$ denote the vertices
of the $i$th copy of $G_2$, with the understanding that $u_j^i$ is
the copy of $u_j$ for each $j$. Denote
$W_j=\left\{u_j^1,u_j^2,\ldots,u_j^{n_1}\right\}$ for $j = 1, 2,
\ldots, n_2$.
Then $$V(G_1)\cup I(G_1) \cup \left[W_1\cup W_2\cup\cdots\cup W_{n_2}\right]$$ is a
partition of $V(G_1\odot G_2)$. It is clear that the degrees of the
vertices of $G_1\odot G_2$ are: $d_{G_1\odot G_2}(v_i)= n_2+d_{G_1}(v_i)$ for $i=1,2,\ldots,n_1$, $d_{G_1\odot G_2}(e_i) = 2$ for $i=1,2,\ldots,m_1$, and
$d_{G_1\odot G_2}(u_j^i)=d_{G_2}(u_j)+1$ for $i=1,2,\ldots,n_1,\, j=1,2,\ldots,n_2$.

\subsection{$A$-spectra of subdivision-vertex coronae}
\label{sec:SVCoronaA}

\begin{thm}\label{CoronaTh1}
Let $G_1$ be an $r_1$-regular graph on $n_1$ vertices and $m_1$ edges, and $G_2$ an arbitrary graph on $n_2$ vertices. Then
\begin{eqnarray*}
\phi\left(A(G_1\odot G_2);x\right)=x^{m_1-n_1}\cdot\big(\phi(A(G_2);x)\big)^{n_1}\cdot\prod_{i=1}^{n_1}\Big(x^2-\Gamma_{A(G_2)}(x)x-r_1-\lambda_i(G_1)\Big). \end{eqnarray*}
\end{thm}

\begin{proof}
Let $R$ be the incidence matrix \cite{kn:Brouwer12} of $G_1$. Then, with respect to the partition $V(G_1)\cup I(G_1)\cup \left[W_1\cup W_2\cup\cdots\cup W_{n_2}\right]$ of $V(G_1\odot G_2)$, the adjacency matrix of $G_1\odot G_2$ can be written as
\[A(G_1\odot G_2)=\bmat{
                          0_{n_1\times n_1} & R & \mathbf{1}_{n_2}^T\otimes I_{n_1} \\[0.2cm]
                          R^T & 0_{m_1\times m_1} & 0_{m_1\times n_1n_2}\\[0.2cm]
                          \mathbf{1}_{n_2}\otimes I_{n_1} & 0_{n_1n_2\times m_1} & A(G_2)\otimes I_{n_1}
                       },\]
where  $0_{s\times t}$ denotes the $s\times t$ matrix with all entries equal to zero, $I_n$ is the identity matrix of order n.
Thus the adjacency characteristic polynomial of $G_1\odot G_2$ is given by
\begin{eqnarray*}
\phi\left(A(G_1\odot G_2)\right)
&=& \det\bmat{
                          xI_{n_1} & -R & -\mathbf{1}_{n_2}^T\otimes I_{n_1} \\[0.2cm]
                          -R^T & xI_{m_1} & 0_{m_1\times n_1n_2}\\[0.2cm]
                          -\mathbf{1}_{n_2}\otimes I_{n_1} & 0_{n_1n_2\times m_1} & (xI_{n_2}-A(G_2))\otimes I_{n_1}
                       }\\ [0.2cm]
&=& \det((xI_{n_2}-A(G_2))\otimes I_{n_1})\cdot\det(S),
\end{eqnarray*}
where
\begin{eqnarray*}
S&=&\pmat{
          xI_{n_1} & -R \\[0.2cm]
         -R^T & xI_{m_1}
       }-\pmat{-\mathbf{1}_{n_2}^T\otimes I_{n_1}\\[0.2cm]
                0_{m_1\times n_1n_2}}((xI_{n_2}-A(G_2))\otimes I_{n_1})^{-1}\pmat{-\mathbf{1}_{n_2}\otimes I_{n_1} & 0_{n_1n_2\times m_1}}\\ [0.2cm]
&=&\pmat{
          \left(x-\Gamma_{A(G_2)}(x)\right)I_{n_1}  & -R \\[0.2cm]
         -R^T & xI_{m_1}
       }
\end{eqnarray*}
is the Schur complement\cite{kn:Schur} of $(xI_{n_2}-A(G_2))\otimes I_{n_1}$. It is well known \cite{kn:Cvetkovic10} that $RR^T=A(G_1)+r_1I_{n_1}$. Thus, the result follows from
\begin{eqnarray*}
\det\left(\left(xI_{n_2}-A(G_2)\right)\otimes I_{n_1}\right)&=&\left(\det(xI_{n_2}-A(G_2))\right)^{n_1}\left(\det (I_{n_1})\right)^{n_2}\\
                                                            &=&\big(\phi\left(A(G_2)\right)\big)^{n_1}
\end{eqnarray*}
and
\begin{eqnarray*}
\det (S)&=&x^{m_1}\cdot\det\left(\left(x-\Gamma_{A(G_2)}(x)\right)I_{n_1}-\frac{1}{x}RR^T\right) \\
&=& x^{m_1}\cdot\prod_{i=1}^{n_1}\left(x-\Gamma_{A(G_2)}(x)-\frac{r_1}{x}-\frac{1}{x}\lambda_i(G_1)\right) \\
&=& x^{m_1-n_1}\cdot\prod_{i=1}^{n_1}\Big(x^2-\Gamma_{A(G_2)}(x)x-r_1-\lambda_i(G_1)\Big).
\end{eqnarray*}\qed
\end{proof}

Theorem \ref{CoronaTh1} enables us to compute the $A$-spectra of many subdivision-vertex coronae, if we can determine the $A(G_2)$-coronal $\Gamma_{A(G_2)}(x)$. Fortunately, we have known the $A(G_2)$-coronal for some graph $G_2$. For example, if $G_2$ is an $r_2$-regular graph on $n_2$ vertices, then \cite{kn:McLeman11,kn:Cui12} $\Gamma_{A(G_2)}(x)=n_2/(x-r_2)$,
and if $G_2\cong  K_{p,q}$ which is the complete bipartite graph with $p,q\geqslant1$ vertices in the two parts of its bipartition, then \cite{kn:McLeman11} $\Gamma_{A(G_2)}(x) =((p+q)x+2pq)/(x^2-pq)$.
Thus, Theorem \ref{CoronaTh1} implies the following results immediately.

\begin{cor}\label{ACorRegular1}
Let $G_1$ be an $r_1$-regular graph on $n_1$ vertices and $m_1$ edges, and $G_2$ an $r_2$-regular graph on $n_2$ vertices. Then the $A$-spectrum of $G_1\odot G_2$ consists of:
\begin{itemize}
  \item[\rm (a)] $\lambda_i(G_2)$, repeated $n_1$ times, for each $i=2,3,\ldots,n_2$;
  \item[\rm (b)] $0$, repeated $m_1-n_1$ times;
  \item[\rm (c)] three roots of the equation
\[x^3-r_2x^2-(r_1+\lambda_j(G_1)+n_2)x+r_2(r_1+\lambda_j(G_1))=0\]
for each $j=1, 2, \ldots, n_1$.
\end{itemize}
\end{cor}

\begin{cor}\label{ACorRegular2}
Let $G$ be an $r$-regular graph on $n$ vertices and $m$ edges with $m \geqslant n$, and let $p,q \geqslant 1$ be integers.  Then the $A$-spectrum of $G\odot K_{p,q}$ consists of:
\begin{itemize}
  \item[\rm (a)] $0$, repeated $m+(p+q-3)n$ times;
  \item[\rm (b)] four roots of the equation
\[x^4-(pq+p+q+r+\lambda_j(G))x^2-2pqx+pq(r+\lambda_j(G))=0\]
for each $j=1, 2, \ldots, n$.
\end{itemize}
\end{cor}

A graph whose $A$-spectrum consists of entirely of integers is called an \emph{$A$-integral graph}. The question of ``Which graphs have $A$-integral spectra?'' was first posed by \textsc{F. Harary} and \textsc{A.J. Schwenk} in 1973 (see \cite{Harary74}), with the immediate remark that the general problem appears challenging and intractable. Although not only the number of $A$-integral graphs is infinite, but also one can find them in all classes of graphs and among graphs of all orders, $A$-integral graphs are very rare and difficult to be found. For more properties and constructions on $A$-integral graphs, please refer to an excellent survey \cite{kn:Balinska02}.

It is well known that the complete graph $K_n$ and complete bipartite graph $K_{n,n}$ are $A$-integral graphs. In the following, with the help of $K_n$ and $K_{n,n}$, we will give two constructions of infinite families of $A$-integral graphs by using the subdivision-vertex coronae.

The \emph{complement} $\overline{G}$ of a graph $G$ is the graph with the same vertex set as $G$ such that two vertices are adjacent in $\overline{G}$ if and only if they are not adjacent in $G$. Note that the complete graph $K_n$ is $(n-1)$-regular with the $A$-spectrum $(n-1)^1,\,(-1)^{n-1}$, where $a^b$ denotes that the multiplicity of $a$ is $b$. Then, by Corollary \ref{ACorRegular1}, the $A$-spectrum of $K_{n_1}\odot \overline{K_{n_2}}$ consists of $\left(\pm\sqrt{2n_1-2+n_2}\right)^1$, $\left(\pm\sqrt{n_1-2+n_2}\right)^{n_1-1}$, $0^{n_1(n_2-1)+m_1}$, which implies that $K_{n_1}\odot \overline{K_{n_2}}$ is $A$-integral if and only if $\sqrt{2n_1-2+n_2}$ and $\sqrt{n_1-2+n_2}$ are integers. Now we present our first construction of an infinite family of $A$-integral graphs (see Figure \ref{fg22} for an example).

\begin{figure}
\centering
\vspace{-1.75cm}
\includegraphics*[height=8.0cm]{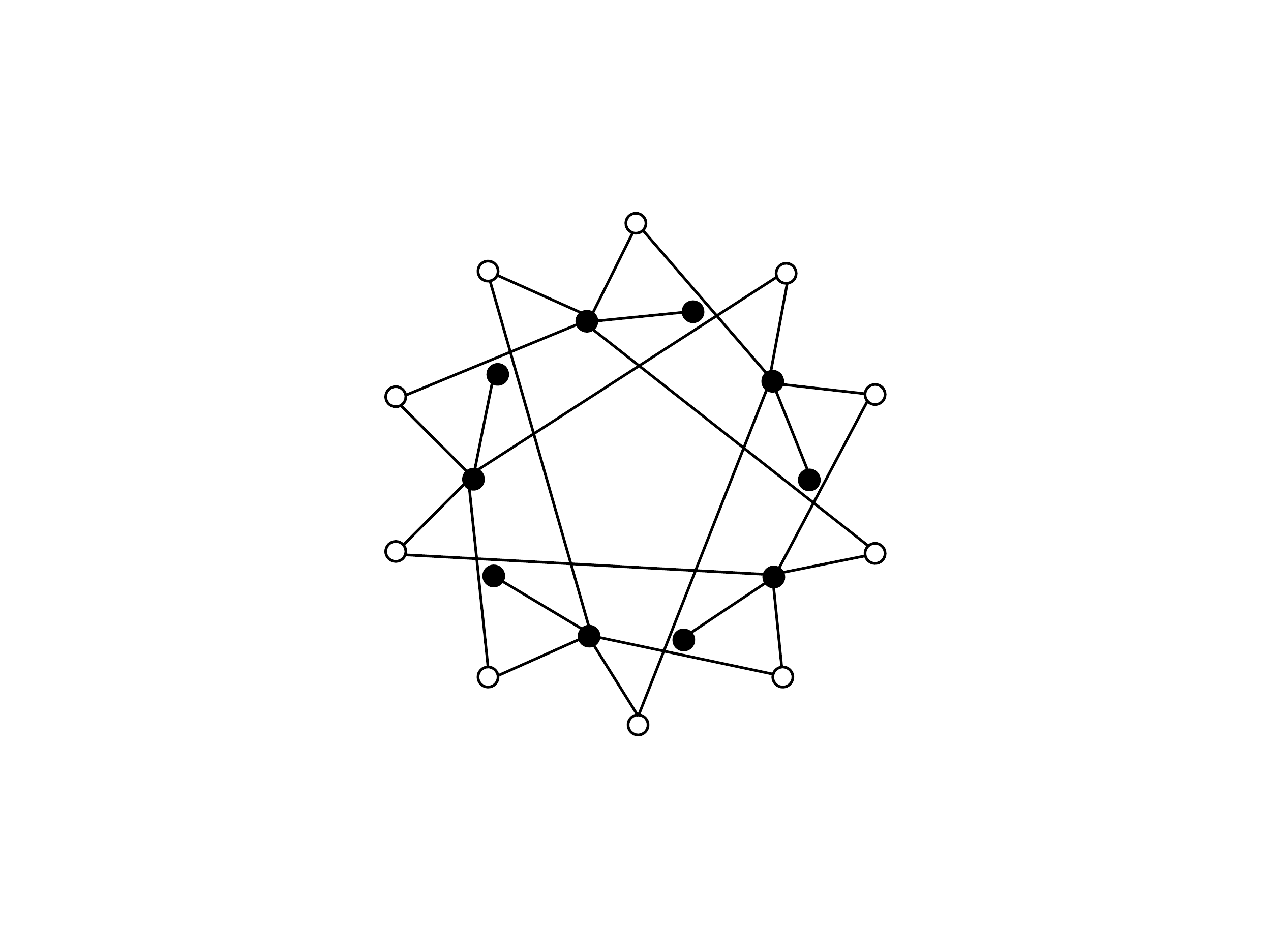}
\vspace{-1.75cm}
\caption{\small $K_5\odot \overline{K_1}$ with $A$-spectrum $(\pm3)^1$, $(\pm2)^4$, $0^{10}$.}
\label{fg22}
\end{figure}

\begin{cor}\label{ACorRegular17}
$K_{n_1}\odot \overline{K_{n_2}}$ is $A$-integral if and only if $n_1=s^2-h^2$ and $n_2=2h^2-s^2+2$ for $h=2,3,\ldots$,  $s=3,4,\ldots$, and $h^2<s^2<2h^2+2$.
\end{cor}

\begin{proof}
From the above statement, $K_{n_1}\odot \overline{K_{n_2}}$ is $A$-integral if and only if $\sqrt{2n_1-2+n_2}$ and $\sqrt{n_1-2+n_2}$ are integers.  Let $\sqrt{2n_1-2+n_2}=s$ and $\sqrt{n_1-2+n_2}=h$, where $s, h$ are nonnegative integers. Solving these two equations, we obtain that $ n_1 = s^2 - h^2 $, $n_2 = 2h^2 - s^2 + 2$. Since $n_1 > 0 $ and $n_2 > 0$,  we have that $h\ge2$, $s\ge3$ and $h^2<s^2<2h^2+2$.
\qed
\end{proof}

Notice that the complete bipartite graph $K_{n,n}$ is $n$-regular with the $A$-spectrum $(\pm n)^1,\,0^{2n-2}$. By Corollary \ref{ACorRegular1}, the $A$-spectrum of $K_{n_1,n_1}\odot \overline{K_{n_2}}$ consists of $\left(\pm\sqrt{2n_1+n_2}\right)^1$, $\left(\pm\sqrt{n_1+n_2}\right)^{2n_1-2}$,  $\left(\pm\sqrt{n_2}\right)^1$,  $0^{2n_1n_2-2n_1+m_1}$, which implies that $K_{n_1,n_1}\odot \overline{K_{n_2}}$ is $A$-integral if and only if $\sqrt{2n_1+n_2}$, $\sqrt{n_1+n_2}$ and $\sqrt{n_2}$ are integers. Here, we present our second construction of an infinite family of $A$-integral graphs.

\begin{cor}\label{ACorRegular18d}
$K_{n_1,n_1}\odot \overline{K_{n_2}}$ is $A$-integral if $n_1=4st^2(2s^2+3s+1)$ and $n_2=t^2(2s^2-1)^2$ for $s=1,2,3,\ldots$, $t=1,2,3,\ldots$.
\end{cor}

\begin{proof}
From the above statement, $K_{n_1,n_1}\odot \overline{K_{n_2}}$ is $A$-integral if and only if $\sqrt{2n_1+n_2}$ , $\sqrt{n_1+n_2}$ and $\sqrt{n_2}$ are integers.  Let $\sqrt{2n_1+n_2}=a$, $\sqrt{n_1+n_2}=b$ and $\sqrt{n_2}=c$ , where $a, b , c$ are nonnegative integers. Solving these equations, we obtain that $n_1=a^2-b^2$ and $n_2=2b^2-a^2=c^2$. Notice that $2b^2-a^2=c^2$ is equivalent to $(a+b)(a-b)=(b+c)(b-c)$. Let $u=a+b$, $v=a-b$, $x=b+c$ and $y=b-c$. Then we obtain that $uv=xy$ and $u-v=x+y$. Combining these two equations and eliminating $u$, we have $2v^2=(x-v)(y-v)$. Let $x-v=2vs$ and $s(y-v)=v$ for $s=1,2,\ldots$. Then $x=2vs+v=b+c$ and $y=\frac{v}{s}+v=b-c$. Thus, $a=b+v=vs+2v+\frac{v}{2s}$, $b=vs+v+\frac{v}{2s}$ and $c=vs-\frac{v}{2s}$. Let $v=2st$ for $t=1,2,\ldots$. Then we have $a=2s^2t+4st+t$, $b=2s^2t+2st+t$ and $c=2s^2t-t$, which imply that $n_1=4st^2(2s^2+3s+1)$ and $n_2=c^2=t^2(2s^2-1)^2$.
\qed
\end{proof}

As stated in the following corollary of Theorem \ref{CoronaTh1}, the subdivision-vertex corona enables us to obtain infinitely many pairs of $A$-cospectral graphs.

\begin{cor}\label{ACorRegular3}
\begin{itemize}
  \item[\rm (a)] If $G_1$ and $G_2$ are $A$-cospectral regular graphs, and $H$ is an arbitray graph, then $G_1\odot H$ and $G_2\odot H$ are $A$-cospectral.
  \item[\rm (b)] If G is a regular graph, and $H_1$ and $H_2$ are $A$-cospectral graphs with $\Gamma_{A(H_1)}(x)=\Gamma_{A(H_2)}(x)$, then $G\odot H_1$ and $G\odot H_2$ are $A$-cospectral.
\end{itemize}
\end{cor}

\subsection{$L$-spectra of subdivision-vertex coronae}
\label{sec:SVCoronaL}

\begin{thm}\label{CoronaTh2}
Let $G_1$ be an $r_1$-regular graph on $n_1$ vertices and $m_1$ edges, and $G_2$ an arbitrary graph on $n_2$ vertices. Then
\begin{eqnarray*}
%\begin{split}
\begin{aligned}
\phi\left(L(G_1 \odot G_2);x\right)
&=(x-2)^{m_1-n_1}\cdot\prod_{i=2}^{n_2}\big(x-1-\mu_i(G_2)\big)^{n_1} \\
&\quad\cdot\prod_{i=1}^{n_1}\Big(x^3-(3+r_1+n_2)x^2+(2+r_1+\mu_i(G_1)+2n_2)x-\mu_i(G_1)\Big).
\end{aligned}
%\end{split}
\end{eqnarray*}
\end{thm}

\begin{proof}
Let $R$ be the incidence matrix \cite{kn:Brouwer12} of $G_1$. Since $G_1$ is an $r_1$-regular graph, we have $D(G_1)=r_1I_{n_1}$. Then the Laplacian matrix of $G_1\odot G_2$ can be written as
\[L(G_1\odot G_2)=\bmat{
                          (r_1+n_2)I_{n_1} & -R & -\mathbf{1}_{n_2}^T\otimes I_{n_1} \\[0.2cm]
                          -R^T & 2I_{m_1} & 0_{m_1\times n_1n_2}\\[0.2cm]
                         - \mathbf{1}_{n_2}\otimes I_{n_1} & 0_{n_1n_2\times m_1} & (L(G_2)+I_{n_2})\otimes I_{n_1}
                       },\]
Thus the Laplacian characteristic polynomial of $G_1\odot G_2$ is given by
\begin{eqnarray*}
\begin{aligned}
\phi\left(L(G_1\odot G_2)\right)
 &=\det\bmat{
                          (x-r_1-n_2)I_{n_1} & R & \mathbf{1}_{n_2}^T\otimes I_{n_1} \\[0.2cm]
                          R^T & (x-2)I_{m_1} & 0_{m_1\times n_1n_2}\\[0.2cm]
                          \mathbf{1}_{n_2}\otimes I_{n_1} & 0_{n_1n_2\times m_1} & ((x-1)I_{n_2}-L(G_2))\otimes I_{n_1}
                       } \\ %[0.2cm]
 &=\det(((x-1)I_{n_2}-L(G_2))\otimes I_{n_1})\cdot\det(S) ,
\end{aligned}
\end{eqnarray*}
where
\begin{eqnarray*}
S=\pmat{
          \left(x-r_1-n_2-\Gamma_{L(G_2)}(x-1)\right)I_{n_1}  & R \\[0.2cm]
         R^T & (x-2)I_{m_1}
       }
\end{eqnarray*}
is the Schur complement\cite{kn:Schur} of $((x-1)I_{n_2}-L(G_2))\otimes I_{n_1}$. Notice that $\lambda_i(G_1)=r_1-\mu_i(G_1), i=1,2,\ldots,n_1$. Then the result follows from
\begin{eqnarray*}
\begin{aligned}
\det\big(((x-1)I_{n_2}-L(G_2))\otimes I_{n_1}\big) &= \prod_{i=1}^{n_2}\big(x-1-\mu_i(G_2)\big)^{n_1},
\end{aligned}
\end{eqnarray*}
and
\begin{eqnarray*}
\begin{aligned}
\det (S) &=(x-2)^{m_1}\cdot\det\left(\left(x-r_1-n_2-\Gamma_{L(G_2)}(x-1)\right)I_{n_1}-\frac{1}{x-2}RR^T\right)\\
&=(x-2)^{m_1}\cdot\prod_{i=1}^{n_1}\left(x-r_1-n_2-\Gamma_{L(G_2)}(x-1)-\frac{\lambda_i(G_1)+r_1}{x-2}\right) \\
&=(x-2)^{m_1-n_1}\cdot\prod_{i=1}^{n_1}\Big(x^2-(2+r_1+n_2+\Gamma_{L(G_2)}(x-1))x+2n_2+2\Gamma_{L(G_2)}(x-1)+\mu_i(G_1)\Big). \\
&=(x-1)^{-n_1}\cdot(x-2)^{m_1-n_1}\cdot\prod_{i=1}^{n_1}\Big(x^3-(3+r_1+n_2)x^2+(2+r_1+\mu_i(G_1)+2n_2)x-\mu_i(G_1)\Big). \\
\end{aligned}
\end{eqnarray*}
\qed
\end{proof}

%Theorem \ref{CoronaTh2} implies the following result.
%\begin{cor}\label{ACorRegular4}
%Let $G_1$ be an $r_1$-regular graph on $n_1$ vertices and $m_1$ edges, and $G_2$ an $r_2$-regular graph on $n_2$ vertices. Then the $L$-spectrum of $G_1\odot G_2$ consists of:
%\begin{itemize}
%  \item[\rm (a)] $\mu_i(G_2)+1$, repeated $n_1$ times, for each $i=2,3,\ldots,n_2$;
%  \item[\rm (b)] $2$, repeated $m_1-n_1$ times;
% \item[\rm (c)] three roots of the equation
%\[x^3-(3+r_1+n_2)x^2+\big(2+r_1+\mu_j(G_1)+2n_2\big)x-\mu_j(G_1)=0\]
%for each $j=1, 2, \ldots, n_1$.
%\end{itemize}
%\end{cor}

%\begin{proof}
%(a) Since $G_2$ is an $r_2$-regular graph on $n_2$ vertices, by (\ref{eq:GammaTL}) we have
%$$
%\Gamma_{L(G_2)}(x-1) = \frac{n_2}{x-1}.
%$$
%The only pole of $\Gamma_{L(G_2)}(x-1)$ is $x-1=0$, which is equal to $\mu_1(G_2)$. Thus, by Theorem \ref{CoronaTh2}, for each $i=2,3,\ldots,n_2$, $\mu_i(G_2)+1$ is an eigenvalue of $G_1\odot G_2$ repeated $n_1$ times.

%(b) This case can be obtained readily from Theorem \ref{CoronaTh2}.

%(c) The remaining $3n_1$ eigenvalues of $G_1\odot G_2$ are obtained by solving
%\[x^2-(2+r_1+n_2+\frac{n_2}{x-1})x+2n_2+\frac{2n_2}{x-1}+\mu_j(G_1)=0\]
%for each $j=1,2,\ldots,n_1$ .
%\qed
%\end{proof}

Theorem \ref{CoronaTh2} helps us to construct infinitely many pairs of $L$-cospectral graphs.

\begin{cor}\label{ACorRegular6}
\begin{itemize}
  \item[\rm (a)] If $G_1$ and $G_2$ are $L$-cospectral regular graphs, and $H$ is an arbitrary graph, then $G_1\odot H$ and $G_2\odot H$ are $L$-cospectral.
  \item[\rm (b)] If G is a regular graph, and $H_1$ and $H_2$ are $L$-cospectral graphs, then $G\odot H_1$ and $G\odot H_2$ are $L$-cospectral.
\end{itemize}
\end{cor}

Let $\emph{t}(\emph{G})$ denote the number of spanning trees of $\emph{G}$. It is well known \cite{kn:Cvetkovic95} that if $\emph{G}$ is a connected graph on \emph{n} vertices with Laplacian spectrum $0=\mu_1(G)<\mu_2(G)\leq \cdots \leq \mu_n(G)$, then
$$
t(G)=\frac {\mu_2(G) \cdots \mu_n(G)}{n}.
$$
By Theorem \ref{CoronaTh2}, we can readily obtain the following result.

\begin{cor}\label{ACorRegular5}
Let $G_1$ be an $r_1$-regular graph on $n_1$ vertices and $m_1$ edges, and $G_2$ an  arbitrary graph on $n_2$ vertices. Then
$$
t(G_1\odot G_2)=\frac{2^{m_1-n_1}\cdot(2+r_1+2n_2)n_1\cdot t(G_1)\cdot\prod_{i=2}^{n_2}\big(\mu_i(G_2)+1\big)^{n_1}}{n_1+m_1+n_1n_2}.
$$
\end{cor}

The \emph{Kirchhoff index} of a graph $G$, denoted by $Kf(G)$, is defined as the sum of resistance distances between all pairs of vertices \cite{kn:Bonchev94,kn:Klein93}. At almost exactly the same time, Gutman et al. \cite{kn:Gutman96} and Zhu et al. \cite{kn:zhu96} proved that the Kirchhoff index of a connected graph $G$ with $n\,(n\geq2)$ vertices can be expressed as $$Kf(G)= n\sum_{i=2}^{n}\frac{1}{\mu_i(G)},$$
where $\mu_2(G), \ldots, \mu_n(G)$ are the non-zero Laplacian eigenvalues of $G$. Theorem \ref{CoronaTh2} implies the following result.

\begin{cor}\label{ACorRegular51}
Let $G_1$ be an $r_1$-regular graph on $n_1$ vertices and $m_1$ edges, and $G_2$ an arbitrary graph on $n_2$ vertices. Then
\begin{eqnarray*}
\begin{aligned}
Kf(G_1\odot G_2) &= \Big(n_1(1+n_2)+m_1\Big)\\
&\quad\times\left(\frac{m_1+n_1-2}{2}+\frac{3+r_1+n_2}{2+r_1+2n_2}+\frac{2+r_1+2n_2}{n_1}\cdot Kf(G_1)+\sum_{i=2}^{n_2}\frac{n_1}{1+\mu_i(G_2)}\right).
\end{aligned}
\end{eqnarray*}
\end{cor}

\subsection{$Q$-spectra of subdivision-vertex coronae}
\label{sec:SVCoronaQ}

\begin{thm}\label{CoronaTh3}
Let $G_1$ be an $r_1$-regular graph on $n_1$ vertices and $m_1$ edges, and $G_2$ an arbitrary graph on $n_2$ vertices. Then
\begin{eqnarray*}
\begin{aligned}
\phi\left(Q(G_1 \odot G_2);x\right)
&=(x-2)^{m_1-n_1}\cdot\prod_{i=1}^{n_2}\big(x-1-\nu_i(G_2)\big)^{n_1}\\
&\quad\cdot\prod_{i=1}^{n_1}\Big(x^2-(2+r_1+n_2+\Gamma_{Q(G_2)}(x-1))x+2(r_1+n_2+\Gamma_{Q(G_2)}(x-1))-\nu_i(G_1)\Big).\\
\end{aligned}
\end{eqnarray*}
\end{thm}

\begin{proof}
Let $R$ be the incidence matrix \cite{kn:Brouwer12} of $G_1$. Then the signless Laplacian matrix of $G_1\odot G_2$ can be written as
\[Q(G_1\odot G_2)=\bmat{
                          (r_1+n_2)I_{n_1} & R & \mathbf{1}_{n_2}^T\otimes I_{n_1} \\[0.2cm]
                          R^T & 2I_{m_1} & 0_{m_1\times n_1n_2}\\[0.2cm]
                          \mathbf{1}_{n_2}\otimes I_{n_1} & 0_{n_1n_2\times m_1} & (Q(G_2)+I_{n_2})\otimes I_{n_1}
                       }.\]
The rest of the proof is similar to that of Theorem \ref{CoronaTh2} and hence we omit details.
\qed
\end{proof}
%\begin{proof}
%(a)Since $G_2$ is $r_2$-regular on $n_2$ vertices, by \cite{kn:McLeman11,kn:Cui12} we have
%$$
%\Gamma_{Q(G_2)}(x-1) = \frac{n_2}{x-1-2r_2}.
%$$
%The only pole of $\Gamma_{Q(G_2)}(x-1)$ is $x-1=2r_2$, which is equal to $\nu_{n_2}(G_2)$. Thus, by Theorem \ref{CoronaTh6}, for each $i=1,2,\ldots,n_2-1$, $\nu_i(G_2)+1$ is an eigenvalue of $G_1\circleddash G_2$ repeated $m_1$ times.

%(b)We can obtain $3n_1$ eigenvalues of $G_1\circleddash G_2$ by solving
%\[x^2-(2+r_1+n_2+\frac{n_2}{x-1-2r_2})x+(2+n_2+\frac{n_2}{x-1-2r_2})r_1-\nu_j(G_1)=0\]
%for each $j=1,2,\ldots,n_1$.

%(c)The remaining  eigenvalues of $G_1\circleddash G_2$ are obtained by solving
%\[x-2-n_2-\frac{n_2}{x-1-2r_2}=0\]
%for each repeated $m_1-n_1$ times.
%\qed
%\end{proof}
By applying (\ref{eq:GammaT}), Theorem \ref{CoronaTh3} implies the following result.

\begin{cor}\label{ACorRegular7}
Let $G_1$ be an $r_1$-regular graph on $n_1$ vertices and $m_1$ edges, and $G_2$ an $r_2$-regular graph on $n_2$ vertices. Then
\begin{eqnarray*}
%\begin{split}
\begin{aligned}
\phi\left(Q(G_1 \odot G_2);x\right)
=(x-2)^{m_1-n_1}\cdot\prod_{i=1}^{n_2}\big(x-1-\nu_i(G_2)\big)^{n_1}
\cdot\prod_{i=1}^{n_1}\Big(x^3-ax^2+bx+c\Big).\\
\end{aligned}
%\end{split}
\end{eqnarray*}
where $a=3+r_1+2r_2+n_2$, $b=2+3r_1+2r_1r_2+4r_2+2r_2n_2+2n_2-\nu_i(G_1)$, $c=-2r_1-4r_1r_2-4r_2n_2+\nu_i(G_1)+2r_2\nu_i(G_1)$.
\end{cor}
Theorem \ref{CoronaTh3} enables us to construct infinitely many pairs of $Q$-cospectral graphs.

\begin{cor}\label{ACorRegular8}
\begin{itemize}
  \item[\rm (a)] If $G_1$ and $G_2$ are $Q$-cospectral regular graphs, and $H$ is any graph, then $G_1\odot H$ and $G_2\odot H$ are $Q$-cospectral.
  \item[\rm (b)] If G is a regular graph, and $H_1$ and $H_2$ are $Q$-cospectral graphs with $\Gamma_{Q(H_1)}(x)=\Gamma_{Q(H_2)}(x)$, then $G\odot H_1$ and $G\odot H_2$ are $Q$-cospectral.
 % \item[\rm (c)] If $G_1$ and $G_2$ are $Q$-cospectral regular graphs, and $H_1$ and $H_2$ are $Q$-cospectral graphs with $\Gamma_Q(H_1)(x)=\Gamma_Q(H_2)(x)$ ,then $G\odot H_1$ and $G\odot H_2$ are $Q$-cospectral.

\end{itemize}
\end{cor}

\section{Spectra of subdivision-edge coronae}\label{SEC:spec}

We label the vertices of $G_1\circleddash G_2$ as follows.
Let $V(G_1)=\left\{v_1,v_2,\ldots,v_{n_1}\right\}$,
$I(G_1)=\left\{e_1,e_2,\ldots,e_{m_1}\right\}$ and
$V(G_2)=\left\{u_1,u_2,\ldots,u_{n_2}\right\}$. For $i = 1, 2,
\ldots, m_1$, let $u_1^i,u_2^i,\ldots,u_{n_2}^i$ denote the vertices
of the $i$th copy of $G_2$, with the understanding that $u_j^i$ is
the copy of $u_j$ for each $j$. Denote
$W_j=\left\{u_j^1,u_j^2,\ldots,u_j^{m_1}\right\}$ for $j = 1, 2,
\ldots, n_2$.
Then $$V(G_1)\cup I(G_1) \cup \left[W_1\cup W_2\cup\cdots\cup W_{n_2}\right]$$ is a
partition of $V(G_1\circleddash G_2)$. It is clear that the degrees
of the vertices of $G_1\circleddash G_2$ are: $d_{G_1\circleddash G_2}(v_i)= d_{G_1}(v_i)$ for $i=1,2,\ldots,n_1$,
$d_{G_1\circleddash G_2}(e_i)= 2+n_2$ for $i=1,2,\ldots,m_1$, and
$d_{G_1\circleddash G_2}(u_j^i)=d_{G_2}(u_j)+1$ for $i=1,2,\ldots,m_1,\, j=1,2,\ldots,n_2$.

\subsection{$A$-spectra of subdivision-edge coronae}
\label{sec:SECoronaA}

\begin{thm}\label{CoronaTh4}
Let $G_1$ be an $r_1$-regular graph on $n_1$ vertices and $m_1$ edges, and $G_2$ an arbitrary graph on $n_2$ vertices. Then
\begin{eqnarray*}
\phi \left(A(G_1 \circleddash G_2);x\right)
=\big(\phi(A(G_2);x)\big)^{m_1}\cdot\big(x-\Gamma_{A(G_2)}(x)\big)^{m_1-n_1}\cdot\prod_{i=1}^{n_1}\Big(x^2-\Gamma_{A(G_2)}(x)x-\lambda_i(G_1)-r_1\Big).
 \end{eqnarray*}
\end{thm}

\begin{proof}
Let $R$ be the incidence matrix \cite{kn:Brouwer12} of $G_1$. Then with the partition $V(G_1)\cup I(G_1) \cup \left[W_1\cup W_2\cup\cdots\cup W_{n_2}\right]$, the adjacency matrix of $G_1 \circleddash G_2$ can be written as
\[A(G_1 \circleddash G_2)=\bmat{
                          0_{n_1\times n_1}   &   R    &  0_{n_1\times m_1n_2}   \\[0.2cm]
                          R^T    &   0_{m_1\times m_1}   &   \mathbf{1}_{n_2}^T\otimes I_{m_1}\\[0.2cm]
                          0_{m_1n_2\times n_1}   &   \mathbf{1}_{n_2}\otimes I_{m_1} &   A(G_2)\otimes I_{m_1}
                       }.\]
Thus the adjacency characteristic polynomial of $G_1 \circleddash G_2$ is given by
\begin{eqnarray*}
\phi\left(A(G_1 \circleddash G_2)\right)
&=& \det \bmat{
                          xI_{n_1} & -R & 0_{n_1\times m_1n_2} \\[0.2cm]
                          -R^T & xI_{m_1} & -\mathbf{1}_{n_2}^T\otimes I_{m_1}\\[0.2cm]
                          0_{m_1n_2\times n_1} & -\mathbf{1}_{n_2}\otimes I_{m_1} & (xI_{n_2}-A(G_2))\otimes I_{m_1}
                       }\\ [0.2cm]
&=& \det((xI_{n_2}-A(G_2))\otimes I_{m_1})\cdot\det(S)\\
&=&\big(\phi\left(A(G_2);x\right)\big)^{m_1}\cdot\det(S),
\end{eqnarray*}
where
\begin{eqnarray*}
S&=&\pmat{
          xI_{n_1} & -R \\[0.2cm]
         -R^T & xI_{m_1}
       }-\pmat{0_{n_1\times m_1n_2}\\[0.2cm]
                -\mathbf{1}_{n_2}^T\otimes I_{m_1}}((xI_{n_2}-A(G_2))\otimes I_{n_1})^{-1}
                \pmat{ 0_{m_1n_2\times n_1}  &  -\mathbf{1}_{n_2}\otimes I_{m_1} }\\ [0.2cm]
&=&\pmat{
           xI_{n_1} & -R \\[0.2cm]
           -R^T &\left(x-\Gamma_{A(G_2)}(x)\right)I_{m_1}
       }
\end{eqnarray*}
is the Schur complement\cite{kn:Schur} of $(xI_{n_2}-A(G_2))\otimes I_{m_1}$.

Note that $R^TR=A(\mathcal{L}(G_1))+2I_{m_1}$ \cite{kn:Cvetkovic10}, and the eigenvalues of $\mathcal{L}(G_1)$ are $\lambda_i(G_1)+r_1-2$, for $i=1, 2, \ldots, n_1 $, and $-2$ repeated $m_1-n_1$ times\cite[Theorem 2.4.1]{kn:Cvetkovic10}, where $\mathcal{L}(G)$ denotes the line graph of G. Thus by applying these arguments, the result follows from
\begin{eqnarray*}
\det (S) &=&x^{n_1}\cdot\det\left(\left(x-\Gamma_{A(G_2)}(x)\right)I_{m_1}-\frac{1}{x}{R^TR}\right) \\
&=& x^{n_1}\cdot\prod_{i=1}^{m_1}\left(x-\Gamma_{A(G_2)}(x)-\frac{\lambda_i(\mathcal{L}(G_1))+2}{x}\right) \\
&=&x^{n_1}\cdot\Big(x-\Gamma_{A(G_2)}(x)-\frac{-2+2}{x}\Big)^{m_1-n_1}\cdot\prod_{i=1}^{n_1}\left(x-\Gamma_{A(G_2)}(x)-\frac{\lambda_i(G_1)+r_1-2+2}{x}\right)
\\
&=& \big(x-\Gamma_{A(G_2)}(x)\big)^{m_1-n_1}
 \cdot\prod_{i=1}^{n_1}\Big(x^2-\Gamma_{A(G_2)}(x)x-r_1-\lambda_i(G_1)\Big).
\end{eqnarray*}\qed
\end{proof}

Similar to Corollaries \ref{ACorRegular1}, \ref{ACorRegular2} and \ref{ACorRegular3}, Theorem \ref{CoronaTh4} implies the following results.

\begin{cor}\label{ACorRegular9}
Let $G_1$ be an $r_1$-regular graph on $n_1$ vertices and $m_1$ edges, and $G_2$ an $r_2$-regular graph on $n_2$ vertices. Then the $A$-spectrum of $G_1\circleddash G_2$ consists of:
\begin{itemize}
  \item[\rm (a)] $\lambda_i(G_2)$, repeated $m_1$ times, for each $i=2,3,\ldots,n_2$;
  \item[\rm (b)] two roots of the equation
\[x^2-r_2x-n_2=0,\]
each root repeated $ m_1-n_1$ times;
  \item[\rm (c)] three roots of the equation
\[x^3-r_2x^2-(r_1+\lambda_j(G_1)+n_2)x+r_2(r_1+\lambda_j(G_1))=0\]
for each $j=1, 2, \ldots, n_1$.
\end{itemize}
\end{cor}

%\begin{proof}
% (a) Since $G_2$ is $r_2$-regular with $n_2$ vertices, by \cite{kn:McLeman11,kn:Cui12} we have
% $$
% \Gamma_{A(G_2)}(x) = \frac{n_2}{x-r_2},
% $$
%The only pole of $\Gamma_{A(G_2)}(x)$ is $x=r_2$, which is equal to $\l_1(G_2)$. Thus, by Theorem \ref{CoronaTh4}, for each $i=2,3,\ldots,n_2$, $\lambda_i(G_2)$ is an eigenvalue of $G_1\circleddash G_2$ repeated $m_1$ times.

%(b) We can obtain $3n_1$ eigenvalues of $G_1\circleddash G_2$ by solving
%\[x^2-\frac{n_2}{x-r_2}x-r_1-\lambda_j(G_1)=0\]
%for each $j=1,2,\ldots,n_1$.

%(c) The remaining  eigenvalues of $G_1\circleddash G_2$ are obtained by solving
%\[x-\frac{n_2}{x-r_2}=0\]
%for each repeated $m_1-n_1$ times.
%\qed
%\end{proof}

\begin{cor}\label{ACorRegular10}
Let $G$ be an $r$-regular graph on $n$ vertices and $m$ edges with $m \geqslant n$, and let $p,q \geqslant 1$ be integers.  Then the $A$-spectrum of $G \circleddash K_{p,q}$ consists of :
\begin{itemize}
  \item[\rm (a)] $0$, repeated $m(p+q-2)$ times;
  \item[\rm (b)] three roots of the equation
\[x^3-(pq+p+q)x-2pq=0,\]
each root repeated $m-n$ times;
  \item[\rm (c)] four roots of the equation
\[x^4-(pq+p+q+r+\lambda_i(G))x^2-2pqx+pq(r+\lambda_i(G))=0\]
for each $i=1, 2, \ldots, n$.
\end{itemize}
\end{cor}

\begin{cor}\label{ACorRegular11}
\begin{itemize}
  \item[\rm (a)] If $G_1$ and $G_2$ are $A$-cospectral regular graphs, and $H$ is an arbitrary graph, then $G_1\circleddash H$ and $G_2\circleddash H$ are $A$-cospectral.
  \item[\rm (b)] If G is a regular graph, and $H_1$ and $H_2$ are $A$-cospectral graphs with $\Gamma_{A(H_1)}(x)=\Gamma_{A(H_2)}(x)$, then $G\circleddash H_1$ and $G\circleddash H_2$ are $A$-cospectral.
\end{itemize}
\end{cor}

Similar to Corollary \ref{ACorRegular17}, the subdivision-edge coronae enable us to construct infinite families of $A$-integral graphs by using Corollary \ref{ACorRegular9}. Note that the $A$-spectrum of $K_{n_1}\circleddash \overline{K_{n_2}}$ is $\left(\pm\sqrt{2n_1-2+n_2}\right)^1$, $\left(\pm\sqrt{n_1-2+n_2}\right)^{n_1-1}$, $\left(\pm\sqrt{n_2}\right)^{m_1-n_1}$, $0^{m_1(n_2-1)+n_1}$. Then $K_{n_1}\circleddash \overline{K_{n_2}}$ is $A$-integral if and only if $\sqrt{2n_1-2+n_2}$, $\sqrt{n_1-2+n_2}$ and $\sqrt{n_2}$ are integers.

Corollary \ref{ACorRegular9} implies that the $A$-spectrum of $K_{n_1,n_1}\circleddash \overline{K_{n_2}}$ consists of $\left(\pm\sqrt{2n_1+n_2}\right)^1$, $\left(\pm\sqrt{n_1+n_2}\right)^{2n_1-2}$,  $\left(\pm\sqrt{n_2}\right)^{m_1-2n_1+1}$,  $0^{2n_1+m_1n_2-m_1}$. Thus $K_{n_1,n_1}\circleddash \overline{K_{n_2}}$ is $A$-integral if and only if $\sqrt{2n_1+n_2}$, $\sqrt{n_1+n_2}$ and $\sqrt{n_2}$ are integers. Then we have the following two constructions of $A$-integral graphs (see Figure \ref{fg33} for an example of Corollary \ref{ACorRegular18fd}).

\begin{figure}
\centering
\vspace{-0.7cm}
\includegraphics*[height=6.6cm]{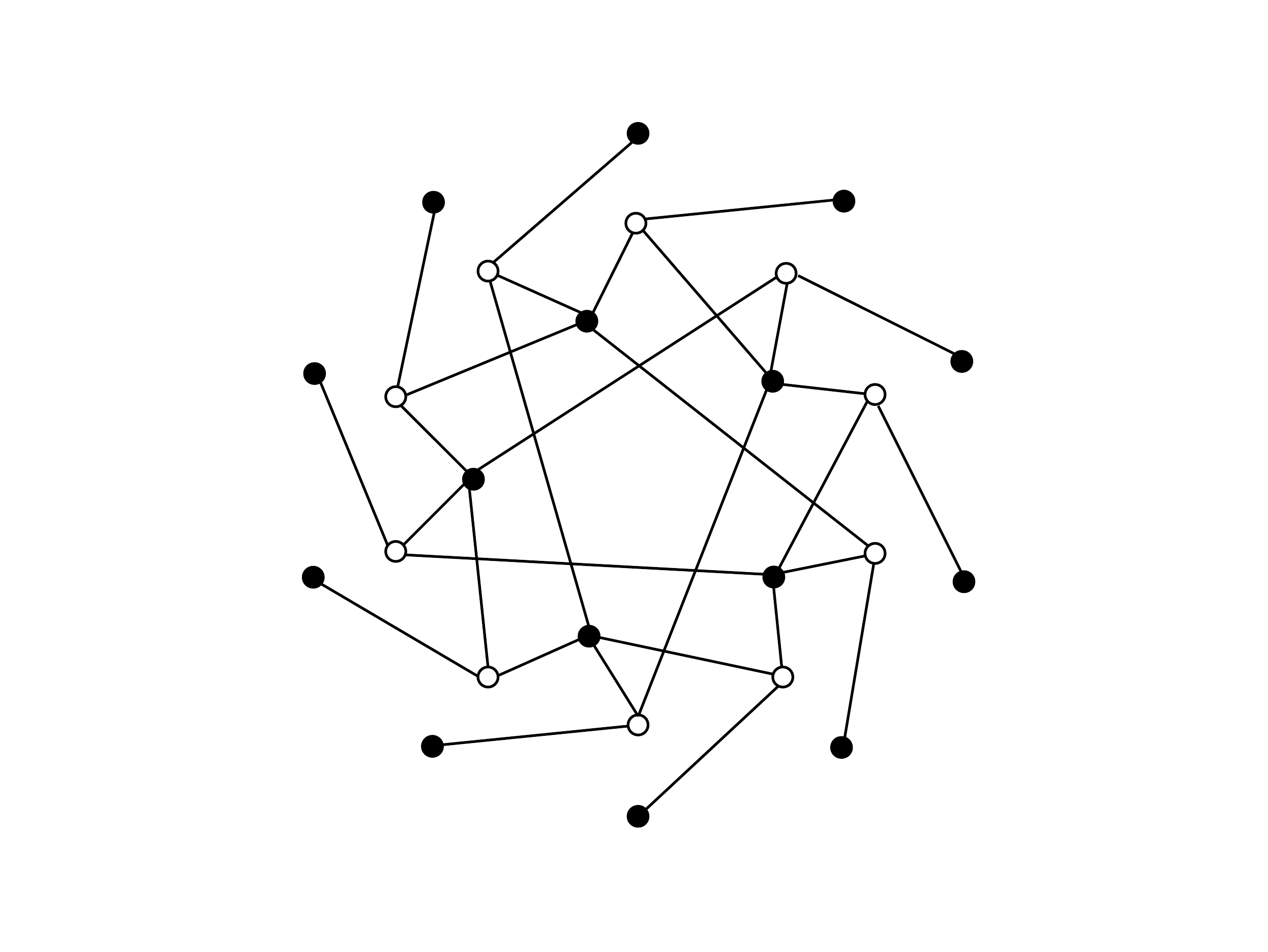}
\vspace{-0.7cm}
\caption{\small $K_5\circleddash \overline{K_1}$ with $A$-spectrum $(\pm3)^1$, $(\pm2)^4$, $(\pm1)^5$, $0^{5}$.}
\label{fg33}
\end{figure}

\begin{cor}\label{ACorRegular18fd}
$K_{n_1}\circleddash \overline{K_{n_2}}$ is $A$-integral if $n_1=2t+3$ and $n_2=t^2$ for $t=1,2,\ldots$.
\end{cor}

\begin{cor}\label{ACorRegular91fd}
$K_{n_1,n_1}\circleddash \overline{K_{n_2}}$ is $A$-integral if $n_1=4st^2(2s^2+3s+1)$ and $n_2=t^2(2s^2-1)^2$ for $s=1,2,3,\ldots$, $t=1,2,3,\ldots$,
\end{cor}

\subsection{$L$-spectra of subdivision-edge coronae}
\label{sec:SECoronaL}

\begin{thm}\label{CoronaTh5}
Let $G_1$ be an $r_1$-regular graph on $n_1$ vertices and $m_1$ edges, and $G_2$ an arbitrary graph on $n_2$ vertices. Then
\begin{eqnarray*}
\begin{aligned}
\phi\left(L(G_1\circleddash G_2);x\right)
&= \big(x^2-(3+n_2)x+2\big)^{m_1-n_1}\cdot\prod_{i=2}^{n_2}\big(x-1-\mu_i(G_2)\big)^{m_1}\\
 &\quad\cdot\prod_{i=1}^{n_1}\Big(x^3-\big(3+r_1+n_2)x^2+(2+r_1+r_1n_2+\mu_i(G_1))x-\mu_i(G_1)\Big) .
\end{aligned}
 \end{eqnarray*}
\end{thm}

\begin{proof}
Let $R$ be the incidence matrix \cite{kn:Brouwer12} of $G_1$. Then the Laplacian matrix of $G_1 \circleddash G_2$ can be written as
\[L(G_1 \circleddash G_2)=\bmat{
                          r_1I_{n_1}   &   -R    &  0_{n_1\times m_1n_2}   \\[0.2cm]
                          -R^T    &   (2+n_2)I_{m_1}   &   -\mathbf{1}_{n_2}^T\otimes I_{m_1}\\[0.2cm]
                          0_{m_1n_2\times n_1}   &   -\mathbf{1}_{n_2}\otimes I_{m_1}  &  ( L(G_2)+I_{n_2})\otimes I_{m_1}
                       }.\]
Thus the Laplacian characteristic polynomial of $G_1 \circleddash G_2$ is given by
\begin{eqnarray*}
\phi\left(L(G_1 \circleddash G_2)\right)
&=& \det\bmat{
                          (x-r_1)I_{n_1} & R & 0_{n_1\times m_1n_2} \\[0.2cm]
                          R^T & (x-2-n_2)I_{m_1} & \mathbf{1}_{n_2}^T\otimes I_{m_1}\\[0.2cm]
                          0_{m_1n_2\times n_1} & \mathbf{1}_{n_2}\otimes I_{m_1} & ((x-1)I_{n_2}-L(G_2))\otimes I_{m_1}
                       }\\ [0.2cm]
&=& \det(((x-1)I_{n_2}-L(G_2))\otimes I_{m_1})\cdot\det(S)\\
&=& \det(S)\cdot\prod_{i=1}^{n_2}\big(x-1-\mu_i(G_2)\big)^{m_1},
\end{eqnarray*}
where
\begin{eqnarray*}
S&=&\pmat{
           (x-r_1)I_{n_1} & R \\[0.2cm]
           R^T &\left(x-2-n_2-\Gamma_{L(G_2)}(x-1)\right)I_{m_1}
       }
\end{eqnarray*}
is the Schur complement\cite{kn:Schur} of $((x-1)I_{n_2}-L(G_2))\otimes I_{m_1}$. Recall that $\mu_i(G_1)=r_1-\lambda_i(G_1)$, $i=1,2,\ldots,n_1$. Then the result follows from

\begin{eqnarray*}
\begin{aligned}
\det (S) &=(x-r_1)^{n_1}\cdot\det\left(\left(x-2-n_2-\Gamma_{L(G_2)}(x-1)\right)I_{m_1}-\frac{1}{x-r_1}{R^TR}\right) \\
&= (x-r_1)^{n_1}\cdot\prod_{i=1}^{m_1}\left(x-2-n_2-\Gamma_{L(G_2)}(x-1)-\frac{\lambda_i(\mathcal{L}(G_1))+2}{x-r_1}\right) \\
&=(x-r_1)^{n_1}\cdot\left(x-2-n_2-\Gamma_{L(G_2)}(x-1)-\frac{-2+2}{x-r_1}\right)^{m_1-n_1}\\
    &\quad\cdot\prod_{i=1}^{n_1}\left(x-2-n_2-\Gamma_{L(G_2)}(x-1)-\frac{\lambda_i(G_1)+r_1-2+2}{x-r_1}\right) \\
&=\big(x-2-n_2-\Gamma_{L(G_2)}(x-1)\big)^{m_1-n_1} \\
  &\quad \cdot\prod_{i=1}^{n_1}\Big(x^2-\big(2+r_1+n_2+\Gamma_{L(G_2)}(x-1)\big)x+r_1n_2+r_1\Gamma_{L(G_2)}(x-1)+\mu_i(G_1)\Big).\\
&=(x-1)^{-m_1}\cdot\big(x^2-(3+n_2)x+2\big)^{m_1-n_1} \\
  &\quad \cdot\prod_{i=1}^{n_1}\Big(x^3-\big(3+r_1+n_2)x^2+(2+r_1+r_1n_2+\mu_i(G_1))x-\mu_i(G_1)\Big).
\end{aligned}
\end{eqnarray*}\qed
\end{proof}

%By (\ref{eq:GammaTL}), Theorem \ref{CoronaTh5} implies the following results.
%\begin{cor}\label{ACorRegular12}
%Let $G_1$ be an $r_1$-regular graph on $n_1$ vertices and $m_1$ edges, and $G_2$ an $r_2$-regular graph on $n_2$ vertices. Then the $L$-spectrum of $G_1\circleddash G_2$ consists of:
%\begin{itemize}
% \item[\rm (a)] $\mu_i(G_2)+1$, repeated $m_1$ times, for each $i=2,3,\ldots,n_2$;
%  \item[\rm (b)] three roots of the equation
%\[x^3-(3+r_1+n_2)x^2+\big(2+r_1+r_1n_2+\mu_j(G_1)\big)x-\mu_j(G_1)=0\]
%for each $j=1, 2, \ldots, n_1$;
%  \item[\rm (c)]  two roots of the equation
%\[x^2-(3+n_2)x+2=0,\]
%each root repeated $m_1-n_1$ times.
%\end{itemize}
%\end{cor}

%\begin{proof}
%(a) Since $G_2$ is $r_2$-regular on $n_2$ vertices, by (\ref{eq:GammaTL}) we have
%$$
%\Gamma_{L(G_2)}(x-1) = \frac{n_2}{x-1}.
%$$
%The only pole of $\Gamma_{L(G_2)}(x-1)$ is $x-1=0$, which is equal to $\mu_1(G_2)$. Thus, by Theorem \ref{CoronaTh5}, for each $i=2,3,\ldots,n_2$, $\mu_i(G_2)+1$ is an eigenvalue of $G_1\circleddash G_2$ repeated $m_1$ times.

%(b) We can obtain $3n_1$ eigenvalues of $G_1\circleddash G_2$ by solving
%\[x^2-(2+r_1+n_2+\frac{n_2}{x-1})x+r_1n_2+r_1\frac{n_2}{x-1}+\mu_i(G_1)=0\]
%for each $i=1,2,\ldots,n_1$.

%(c) The remaining  eigenvalues of $G_1\circleddash G_2$ are obtained by solving
%\[x-2-n_2-\frac{n_2}{x-1}=0\]
%for each repeated $m_1-n_1$ times.
%\qed
%\end{proof}

Theorem \ref{CoronaTh5} implies the following results.

\begin{cor}\label{ACorRegular13}
Let $G_1$ be an $r_1$-regular graph on $n_1$ vertices and $m_1$ edges, and $G_2$ an  arbitrary graph on $n_2$ vertices. Then
%$$
\begin{eqnarray*}
\begin{aligned}
t(G_1\circleddash G_2) &= \frac{2^{m_1-n_1}\cdot(2+r_1+r_1n_2)n_1\cdot t(G_1)\cdot\prod_{i=2}^{n_2}\big(\mu_i(G_2)+1\big)^{m_1}}{n_1+m_1+m_1n_2}.\\
\end{aligned}
\end{eqnarray*}

\begin{cor}\label{ACorRegular133}
Let $G_1$ be an $r_1$-regular graph on $n_1$ vertices and $m_1$ edges, and $G_2$ an arbitrary graph on $n_2$ vertices. Then
\begin{eqnarray*}
\begin{aligned}
Kf(G_1\circleddash G_2) &= \Big(m_1(1+n_2)+n_1\Big)\\
&\quad\times\left(\frac{(3+n_2)m_1-(n_2+1)n_1-2}{2}+\frac{3+r_1+n_2}{2+r_1+r_1n_2}+\frac{2+r_1+r_1n_2}{n_1}\cdot Kf(G_1)\right.\\
&\quad\quad\quad\quad\quad\left.+\sum_{i=2}^{n_2}\frac{m_1}{1+\mu_i(G_2)}\right).
\end{aligned}
\end{eqnarray*}
\end{cor}

\end{cor}
\begin{cor}\label{ACorRegular14}
\begin{itemize}
  \item[\rm (a)] If $G_1$ and $G_2$ are $L$-cospectral regular graphs, and $H$ is an arbitrary graph, then $G_1\circleddash H$ and $G_2\circleddash H$ are $L$-cospectral.
  \item[\rm (b)] If G is a regular graph, and $H_1$ and $H_2$ are $L$-cospectral graphs, then $G\circleddash H_1$ and $G\circleddash H_2$ are $L$-cospectral.
\end{itemize}
\end{cor}

\subsection{$Q$-spectra of subdivision-edge coronae}
\label{sec:SECoronaQ}

\begin{thm}\label{CoronaTh6}
Let $G_1$ be an $r_1$-regular graph on $n_1$ vertices and $m_1$ edges, and $G_2$ an arbitrary graph on $n_2$ vertices. Then
\begin{eqnarray*}
\begin{aligned}
\phi\left(Q(G_1\circleddash G_2);x\right)
&=\big(x-2-n_2-\Gamma_{Q(G_2)}(x-1)\big)^{m_1-n_1}\cdot\prod_{i=1}^{n_2}\big(x-1-\nu_i(G_2)\big)^{m_1}\\
&\quad\cdot\prod_{i=1}^{n_1}\Big(x^2-(2+r_1+n_2+\Gamma_{Q(G_2)}(x-1))x+r_1(2+n_2+\Gamma_{Q(G_2)}(x-1))-\nu_i(G_1)\Big).
\end{aligned}
\end{eqnarray*}
\end{thm}

\begin{proof}
Let $R$ be the incidence matrix \cite{kn:Brouwer12} of $G_1$. Then the signless Laplacian matrix of $G_1 \circleddash G_2$ can be written as
\[Q(G_1 \circleddash G_2)=\bmat{
                          r_1I_{n_1}   &   R    &  0_{n_1\times m_1n_2}   \\[0.2cm]
                          R^T    &   (2+n_2)I_{m_1}   &   \mathbf{1}_{n_2}^T\otimes I_{m_1}\\[0.2cm]
                          0_{m_1n_2\times n_1}   &   \mathbf{1}_{n_2}\otimes I_{m_1}  &  ( Q(G_2)+I_{n_2})\otimes I_{m_1}
                       }.\]
The result follows by refining the arguments used to prove Theorem \ref{CoronaTh5}.\qed
\end{proof}

By applying (\ref{eq:GammaT}) again, Theorem \ref{CoronaTh6} implies the following result.

\begin{cor}\label{ACorRegular15}
Let $G_1$ be an $r_1$-regular graph on $n_1$ vertices and $m_1$ edges, and $G_2$ an $r_2$-regular graph on $n_2$ vertices. Then the $Q$-spectrum of $G_1\circleddash G_2$ consists of:
\begin{itemize}
  \item[\rm (a)] $\nu_i(G_2)+1$, repeated $m_1$ times, for each $i=1,2,\ldots,n_2-1$;
  \item[\rm (b)] three roots of the equation
\[x^3-ax^2+bx+c=0\]
for each $j=1, 2, \ldots, n_1$, where $a=3+r_1+2r_2+n_2, b=2+r_1n_2+2r_1r_2+2r_2n_2+3r_1+4r_2-\nu_j(G_1), c=-2r_1-4r_1r_2-2r_1r_2n_2+\nu_j(G_1)+2r_2\nu_j(G_1)$;
  \item[\rm (c)] two roots of the equation
\[x^2-(3+2r_2+n_2)x+2(1+2r_2+r_2n_2)=0,\]
each root repeated $ m_1-n_1$ times.
\end{itemize}
\end{cor}

%\begin{proof}
%(a) Since $G_2$ is $r_2$-regular graph on $n_2$ vertices, by (\ref{eq:GammaT}) we have
%$$
%\Gamma_{Q(G_2)}(x-1) = \frac{n_2}{x-1-2r_2}.
%$$
%The only pole of $\Gamma_{Q(G_2)}(x-1)$ is $x-1=2r_2$, which is equal to $\nu_{n_2}(G_2)$. Thus, by Theorem \ref{CoronaTh6}, for each $i=1,2,,\ldots,n_2-1$, $\nu_i(G_2)+1$ is an eigenvalue of $G_1\circleddash G_2$ repeated $m_1$ times.

%(b) We can obtain $3n_1$ eigenvalues of $G_1\circleddash G_2$ by solving
%\[x^2-(2+r_1+n_2+\frac{n_2}{x-1-2r_2})x+r_1(2+n_2+\frac{n_2}{x-1-2r_2})-\nu_j(G_1)=0\]
%for each $j=1,2,\ldots,n_1$.

%(c) The remaining  eigenvalues of $G_1\circleddash G_2$ are obtained by solving
%\[x-2-n_2-\frac{n_2}{x-1-2r_2}=0\]
%for each repeated $m_1-n_1$ times.
%\qed
%\end{proof}

Finally, Theorem \ref{CoronaTh6} enables us to construct infinitely many pairs of $Q$-cospectral graphs.

\begin{cor}\label{ACorRegular16}
\begin{itemize}
  \item[\rm (a)] If $G_1$ and $G_2$ are $Q$-cospectral regular graphs, and $H$ is an arbitrary graph, then $G_1\circleddash H$ and $G_2\circleddash H$ are $Q$-cospectral.
  \item[\rm (b)] If G is a regular graph, and $H_1$ and $H_2$ are $Q$-cospectral graphs with $\Gamma_{Q(H_1)}(x)=\Gamma_{Q(H_2)}(x)$, then $G\circleddash H_1$ and $G\circleddash H_2$ are $Q$-cospectral.
\end{itemize}
\end{cor}

\medskip

\noindent \textbf{Acknowledgements}~~The authors would like to thank Dr. Xiaogang Liu for his suggestion of constructing infinite families of $A$-integral graphs by using the subdivision-vertex and subdivision-edge coronae. P. Lu is supported by the Natural Science Foundation of Gansu Province (No.1212RJZA029).

\end{document}